\documentclass[12pt,reqno]{amsart}

\usepackage[latin1]{inputenc}
\usepackage{amsmath}
\usepackage{amsfonts}
\usepackage{amssymb}
\usepackage{graphics}

\usepackage{enumerate}
\usepackage{amssymb,amsmath,amsthm,amscd}
\usepackage{latexsym,verbatim,graphicx,amsfonts}

\usepackage{amscd}
\usepackage{amsmath}
\usepackage{amssymb}
\usepackage{amsthm}
\usepackage{latexsym}
\usepackage{verbatim}
\usepackage{hyperref}

\theoremstyle{plain}
\newtheorem{theorem}{Theorem}[section]
\newtheorem{thm}[theorem]{Theorem}
\newtheorem{prob}[theorem]{Problem}
\newtheorem{corollary}[theorem]{Corollary}
\newtheorem{lem}[theorem]{Lemma}

\theoremstyle{definition}

\newtheorem{defn}[theorem]{Definition}

\theoremstyle{remark}
\newtheorem{rem}[theorem]{Remark}

\newcommand{\R}{\mathbb{R}}
\newcommand{\C}{\mathbb{C}}

\newcommand{\D}{\mathbb{D}}
\newcommand{\N}{\mathbb{N}}

\subjclass[2010]{Primary 47A15; Secondary 30H10, 30H20.}

\begin{document}
\title[$M_{z^k}$]{A $z^k$-invariant subspace without the wandering property}
\author[Seco]{Daniel Seco}
\address{Instituto de Ciencias Matem\'aticas, Calle Nicol\'as Cabrera, UAM, 28049 Madrid,
Spain.} \email{dsf$\underline{\,\,\,}$cm@yahoo.es}
\date{\today}

\begin{abstract}
We study operators of multiplication by $z^k$ in Dirichlet-type
spaces $D_\alpha$. We establish the existence of $k$ and $\alpha$
for which some $z^k$-invariant subspaces of $D_\alpha$ do not
satisfy the wandering property. As a consequence of the proof, any
Dirichlet-type space accepts an equivalent norm under which the
wandering property fails for some space for the operator of
multiplication by $z^k$, for any $k \geq 6$.
\end{abstract}

\maketitle

\section{Introduction}\label{Intro}

In the present paper we will be concerned with (closed) subspaces of
so-called Dirichlet-type spaces, that remain invariant under the
action of the operator of multiplication by $z^k$, for some $k \in
\N$.

\begin{defn}
Let $\alpha \in \R$. We denote by $D_\alpha$ the
\emph{Dirichlet-type space} with parameter $\alpha$, defined as
\begin{eqnarray*}D_\alpha = \{f \in Hol(\D): f(z) = \sum_{k
=0}^{\infty} a_k z^k,\\ \|f\|^2_\alpha := \sum_{k=0}^{\infty}
|a_k|^2 (k+1)^{\alpha} < \infty \}.\end{eqnarray*}
\end{defn}

The particular case when $\alpha=-1$ is denoted by $A^2$, and is
often referred to as \emph{the Bergman space}. It consists of all
holomorphic functions over the unit disc $\D$ with square integrable
modulus with respect to the normalized Lebesgue area measure (with
density denoted $dA(z)$). For a function $f$ with Maclaurin
coefficients $\{a_k\}_{k \in \N}$, the norm satisfies the identity
\[\|f\|_{-1}^2=\sum_{k=0}^{\infty}
\frac{|a_k|^2}{k+1} = \int_{\D} |f(z)|^2dA(z). \] We refer to
\cite{DuS, HKZ} for further information on this space. Other
relevant values of $\alpha$ are 0 and 1, when the spaces are,
respectively, the Hardy space $H^2$ and the Dirichlet space $D$. See
\cite{Dur, Gar} for information on $H^2$ and \cite{ARSW, EFKMR,
Ross} for information on $D$.

The shift $S$ is the operator taking a function $f$ in $D_\alpha$ to
the function $Sf(z) = z f(z)$. The invariant subspaces for the shift
operator have attracted much interest from specialists, in part
because of the connections with the Hilbert invariant subspace
problem (see \cite{ABFP, Par}). We denote by $M_{g}$ the operator of
multiplication by an analytic function $g$. If $M_g$ is bounded we
say that $g$ is a multiplier (of $D_\alpha$). From now on, if the
value of $\alpha$ is fixed, we denote $\{ U \}_{g}$ the smallest
closed subspace of $D_\alpha$ containing the set of functions $U$
which is invariant under the action of the operator $M_{g}$. For
instance, both $A^2$ and $H^2$ have the space of bounded analytic
functions with the supremum norm as their space of multipliers,
while $D_\alpha$ spaces for $\alpha >1$ are closed under
multiplication. Neither of this is true for $D_\alpha$ whenever
$0<\alpha  \leq 1$. We say that $U$ \emph{generates} $M$ (under $g$)
if $\{ U \}_{g} = M$. The present article deals with invariant
subspaces under the action of $M_{z^k}=S^k$ for some $k \in \N$, and
in particular with the question of whether or not some special
subsets generate or not the whole subspace.

In \cite{ARS}, the authors found the correct generalization for
$A^2$ of a classical result of Beurling for the Hardy space $H^2$
(see \cite{Beu}):

\begin{thm}[Aleman, Richter, Sundberg]\label{ARSthm}
Let $M \subset A^2$ be a closed subspace invariant under $S$. Then
\[ M= \{ M \ominus SM \}_z.
\]
\end{thm}

This is, however, a more intricate situation than that of $H^2$
where each invariant subspace is generated by a single function.
Examples of invariant subspaces not generated by a single element
can be found in \cite{Hed}. Theorem \ref{ARSthm} justifies the
following definition:

\begin{defn} Let $g$ be a multiplier in $D_\alpha$, and let $M \subset D_\alpha$ be a closed subspace that is invariant under
the operator $M_g$. We say that $M$ has the \emph{wandering
property} (relative to $g$ in $D_\alpha$) if \[M= \{ M \ominus gM
\}_g.\]\end{defn}

Shimorin (\cite{Shim}) generalized Aleman-Richter-Sundberg's result
to the shift in Dirichlet-type spaces for $-1 \leq \alpha \leq 1$
and to other particular operators. See also \cite{RaR}. In the paper
\cite{Netal}, the failure of the wandering property is shown for the
shift in the spaces $D_\alpha$ for $\alpha \leq -5$, although these
spaces are equipped with a different equivalent norm, arising from
an integral representation. The result for $\alpha < -5$ is an
application of other results by \cite{HZFail}. In the interesting
articles \cite{KLS, CDS}, the authors consider the question of
generalizing Theorem \ref{ARSthm} to other bounded multiplication
operators in $H^2$ and $A^2$ (respectively). In the latter, the
following problem is proposed:

\begin{prob}\label{prob1}
Let $g \in H^{\infty}$ and consider $M_g$ acting on $A^2$. Assume $M
\subset A^2$ is a closed subspace invariant under $M_g$. Is it true
then that $M$ satisfies the wandering property?
\end{prob}

The question has been studied in several papers including \cite{CDS}
and \cite{CaW}, and even though some counterexamples have been
found, the problem remains open even for the case when $g(z)= z^k$
for some $k \in \N$, $k \geq 2$. The latter is mentioned
specifically as an open problem, already on \cite{CDS}.

The purpose of this article is to construct $z^k$-invariant
subspaces of $D_\alpha$ spaces that do not satisfy the wandering
property. By doing this, we solve an analogue of Problem \ref{prob1}
in those spaces. Our main result is the following:

\begin{thm}\label{main}
There exists $\epsilon >0$ such that for all $\alpha \in
(-16-\epsilon, -16 + \epsilon)$ there exists a closed subspace $M
\subset D_\alpha$ invariant under multiplication by $z^6$, but
without the wandering property.
\end{thm}

Notice that the wandering property depends a priori on the choice of
equivalent norm for the ambient space $D_\alpha$, even though the
property of invariance is independent of this choice. This is in
fact, a key feature and as a corollary of the proof of our main
theorem, we will obtain the following result:

\begin{corollary}\label{coro}
Let $k \geq 6$. There exist two polynomials $F_1$ and $F_2$ such
that for each value of $\alpha$, the space $D_\alpha$ admits a
choice of equivalent norm under which the invariant subspace
generated by $F_1$ and $F_2$ under multiplication by $z^k$ does not
have the wandering property.\end{corollary}

We will complete this with a proof that the behavior for large
enough Dirichlet-type spaces (with their usual norms) is consistent
with this:

\begin{thm}\label{9k28}
Let $k \geq 10, k \in \N$. For each $\alpha <
-(5k+\frac{700}{(k-9)^2})$, $D_\alpha$ contains a $z^k$ invariant
subspace without the wandering property.
\end{thm}

Although this result may seem stronger than Theorem \ref{main}, we
consider the techniques used to establish Theorem \ref{main} to be
more capable of dealing with more general problems. Our numerical
evidence suggests that the optimal threshold for each $k \geq 10$
(and in fact, $k \geq 6$) is much higher than that of Theorem
\ref{9k28}, probably increasing with $k$, and staying somewhere
between $-5$ and $0$ for all values of $k$.

Notice that in $H^2$ with the classical norm, the $z^k$ wandering
property holds for $k \geq 1$, since the powers of the shift all
behave like the shift itself (for instance, Shimorin's proof for
$k=1$ works for all $k$).

\begin{rem}
In the forthcoming paper \cite{GGPS}, we will show among other
results that for each $k \geq 1$ and each negative $\alpha \geq
\frac{\log(2)}{\log(k+1)}$, the $D_\alpha$ spaces with their usual
norms do have the $z^k$ wandering property. In fact, for any $\alpha
<0$ and $k \geq 1$, there exists a choice of optimal norm under
which the wandering property holds for $z^k$. Therefore the norm
dependence is absolutely relevant to the problem.
\end{rem}

In Section \ref{promain1}, we provide sufficient conditions on a
pair of polynomial functions $F_1$ and $F_2$ and on the ambient
space $D_\alpha$ for the subspace $M$ generated by $F_1$ and $F_2$
to satisfy the claim in Theorem \ref{main}. Then, in Section
\ref{promain2}, we show how to arrive to a choice of $F_1$ and $F_2$
that satisfies such sufficient conditions, for the case $\alpha =
-16$ and $k=6$. One of the conditions will be open and depend
continuously on the parameters, and as a consequence there will
exist an interval of possible values of $\alpha$ for which the same
counterexample will work, therefore proving Theorem \ref{main}. We
also provide, in Section \ref{pronew}, the proof of Corollary
\ref{coro} and, in Section \ref{pronew2}, that of Theorem
\ref{9k28}. The values of $\alpha$ and $k$ above may seem arbitrary
but the method will give a clear idea of how $k$ needs to be at
least 6. For other values of $\alpha$ we have performed numerical
computations, finding evidence suggesting that for all $\alpha \leq
\alpha_0=-4.999$ there exist invariant subspaces without the
wandering property for $z^6$. Changing $k$ to a larger value
requires a small modification of our method, but it allows to extend
the bound on $\alpha$ to $\alpha_1=-4.2$. We present these numerical
results in Section \ref{numerics}. We will conclude with some
further remarks in Section \ref{rema}.

\section{Sufficient conditions for the failure of the wandering property}\label{promain1}

We denote by $\langle,\rangle$ the inner product in $D_\alpha$, by
$\omega_k=(k+1)^{\alpha}$ and given $h \in Hol(\D)$ and $s \in \N$,
$\hat{h}(s)$ will be the Taylor coefficient (always centered at 0)
of order $s$ of $h$.

Much of what we will do in this Section is applicable to any $k \geq
6$ and any Dirichlet-type space excepting $H^2$ and $D$. Those
exceptions are the only ones in which $\omega_k$ is an affine
function of $k$, which will make certain linear system become
incompatible. Thus we take $\alpha \in \R \backslash \{0,1\}$. Later
we will concentrate on the case $\alpha=-16$ and $k=6$.

The space $M$ we will construct is the smallest closed subspace of
$D_\alpha$ invariant under multiplication by $z^k$ and containing
two functions $F_1$ and $F_2$, which are taken to be polynomials.
That is
\begin{equation}\label{eqn300}M=\{F_1,F_2\}_{z^k}.\end{equation}

Our objective is to show the existence of spaces $M$ as above for
which the wandering property fails, that is, spaces $M$ that satisfy
\begin{equation}\label{eqn301}
\{M \ominus z^k M\}_{z^k} \neq M.
\end{equation}

By the definition of $M$ in \eqref{eqn300}, $M$ is invariant under
$z^k$, and hence $M$ satisfying \eqref{eqn301} may only exist
provided that $M$ contains at least one function not in $\{M \ominus
z^k M\}_{z^k}$. If such a function exists, one of the generators
must be an example. In fact, we will construct $F_1$ and $F_2$ such
that $F_1$ will belong to $(M \backslash \{M \ominus z^k
M\}_{z^k})$.

We could take a very general polynomial function and find what is
exactly needed. This is not our approach, which focus on taking only
a few degrees of freedom so that equations become solvable in finite
time, explicitly.

Let us define $F_1$ by
\begin{equation}\label{defF1}
F_1(z)= \sum_{i=0}^4 a_i z^i + \sum_{i=0}^3 a_{k+i} z^{k+i},
\end{equation}
and $F_2$, by
\begin{equation}\label{defF2}
F_2(z) = \sum_{i=0}^3 b_i z^i + b_5 z^5,
\end{equation}
where $a_i, b_i$ are complex constant coefficients to be determined
later in terms of $k$ and $\alpha$, and where $a_4$ and $b_5$ are
non-null.

\begin{rem}
Notice that $a_4$ and $b_5$ will play a different role than other
parameters and that they are the only coefficients whose degree is
not congruent with 0, 1, 2 or $3 \mod k$.
\end{rem}

This special role of $a_4$ and $b_5$ is justified by the following
result which will allow us to make use of a form of Fourier analysis
on the elements of $M$.

\begin{lem}\label{Lemma1} Let $M$ be as in \eqref{eqn300}, where
$F_1$ and $F_2$ are defined as in \eqref{defF1} and \eqref{defF2},
and let $a_4 \neq 0 \neq b_5$. Then
\[M = \{f(z)=f_1(z^k)F_1(z) + f_2(z^k)F_2(z): f_1, f_2 \in D_\alpha\}.\]
\end{lem}

\begin{proof}
Denote by $M_0$ the right-hand side in the statement of the Lemma.
To see $M_0 \subset M$, notice that $F_1, F_2$ are polynomials and
hence, multipliers in $A^2$. Baring this in mind, let $f \in M_0$ be
given by the functions $f_1$ and $f_2$ in $D_\alpha$ in the same way
as in the description of $M$ in the statement, and denote the
sequence of Taylor polynomials of $f_1$ as $\{p_n\}_{n \in \N}$ and
those of $f_2$ by $\{q_n\}_{n \in \N}$. Then $g_n(z) =
p_n(z^k)F_1(z) + q_n(z^k) F_2(z)$ is in $M$ for all $n \in \N$ and
$g_n$ converges in $D_\alpha$ norm to $f$. Since $M$ is a closed
subspace, $f \in M$.

On the other hand, to see the other inclusion, let $f(z) = f_1(z^k)
F_1(z) + f_2(z^k) F_2(z)$ and suppose, firstly, that $f_1 \notin
D_\alpha$. Then the Taylor coefficients of $f$ of order $kt+4$, $t
\in \N$, must be \[\hat{f}(kt+4)= \hat{f_1}(t) a_4.\] Recall that we
took $a_4 \neq 0$, and by our hypothesis
\[\|f_1\|^2_\alpha = \sum_{t \in \N} |\hat{f_1}(t)|^2 (t+1)^{\alpha}
= + \infty.\] The necessary conclusion is that \[\|f\|^2_{\alpha}
\geq \sum_{s \in k\N +4} |\hat{f}(s)|^2(s+1)^{\alpha} = +\infty,\]
which is a contradiction with the assumption that $f \in M \subset
D_\alpha$. The same argument works for $f_2$ instead of $f_1$
replacing $a_4$ by $b_5$ and $kt+4$ by $kt+5$ where necessary,
showing that if $f_2 \notin D_\alpha$ then $f \notin M$, which
finishes the proof of the Lemma.
\end{proof}

The proof of Lemma \ref{Lemma1} tells us that we can recover, for
any $f \in M$, the sequence $\{\hat{f_1}(t)\}_{t \in \N}$ and
$\{\hat{f_2}(t)\}_{t \in \N}$ that serve as coordinates of $f$ as an
element of $M$. We may think of such coefficients as the
\emph{Fourier coefficients} of $f$, and of $f_1$ and $f_2$ as a form
of \emph{Fourier transform} of $f$. This will be the main role of
$a_4$ and $b_5$ in our construction. Going forward, our next task
will be to understand which functions belong to $M \ominus z^k M$
depending on the values of the coefficients $a_i$ and $b_i$. Suppose
that $f \in M \ominus z^k M$. Then, for any $s \in \N$, $s \geq 1$,
and for $j=1,2$ we must have
\[\langle f , z^{ks}  F_j \rangle =0.\]
This is, in fact, easily seen to be a characterization of the
elements in $M \ominus z^k M$. It turns out to be useful to express
these relations in terms of the Fourier coefficients
$\{\hat{f_i}(t)\}_{n \in \N, i=1,2}$. We introduce the following
notation:
\begin{defn}\label{notation} We will denote
\begin{itemize}
\item[(A1)] \[A_{s,1} = \langle z^{k(s-1)}F_1, z^{ks}F_1 \rangle = \sum_{h=0}^3 \overline{a_h}
a_{k+h} \omega_{ks+h}.\]
\item[(A2)] \[A_{s,2} = \langle z^{ks} F_1 , z^{ks}F_2 \rangle.\]
\item[(A3)] \[A_{s,3} = \|z^{ks} F_1\|^2_{\alpha}.\]
\item[(A4)] \[A_{s,4} = \|z^{ks} F_2\|^2_{\alpha}.\]
\item[(A5)] \[A_{s,5} = \langle z^{k(s-1)}F_1, z^{ks}F_2 \rangle = \sum_{h=0}^3 \overline{b_h}
a_{k+h} \omega_{ks+h}.\]
 \end{itemize}
\end{defn}
The basic property that they will satisfy is a recurrence relation:
\begin{lem}\label{Lemma3}
Let $f(z)= f_1(z^k) F_1(z) + f_2(z^k) F_2(z)$ for some functions
$f_1, f_2 \in D_\alpha$. Then $f \in M \ominus z^k M$ if and only if
for all $s \geq 1$ we have both:
\begin{itemize}
\item[(a)] \begin{eqnarray*}0 = \hat{f_1}(s+1) \overline{A_{s+1,1}} + \hat{f_2}(s+1) \overline{A_{s+1,5}}+ \\ + \hat{f_1}(s)
A_{s,3} + \hat{f_2}(s) \overline{A_{s,2}} +  \hat{f_1}(s-1)
A_{s,1}.\end{eqnarray*}
\item[(b)] \[0 = \hat{f_1}(s) A_{s,2} + \hat{f_2}(s) A_{s,4} + \hat{f_1}(s-1)
A_{s,5}.\]
\end{itemize}
\end{lem}

\begin{proof}
By Lemma \ref{Lemma1}, $f \in M$, so the only condition that needs
checking is that $f$ is orthogonal to the functions $z^{ks} F_j$,
for $s \geq 1$ and $j=1,2$. (a) is equivalent to $f \perp z^{ks}F_1$
whereas (b) is equivalent to $f \perp z^{ks}F_2$, since any other
combinations of Taylor coefficients that may appear on the
expressions for the inner products are equal to zero.
\end{proof}

From now on, until the end of this Section we will concentrate on
giving conditions, in terms of the relevant quantities $A_{s,r}$
(which only depend on $F_1$ and $F_2$), under which the space $M$
will be as in the statement of the Theorem \ref{main}.

The previous Lemma may be interpreted as a recurrence scheme
relating the coefficients of functions in $M \ominus z^k M$. If we
make certain choices of parameters $a_i$ and $b_i$, we will be able
to use these recurrences to our end. The first step is apparently
innocent but later on, it will allow us to forget about the whole
tail of Fourier coefficients of degrees greater or equal to 2:

\begin{lem}\label{Lemma5}
Suppose that $F_1$ and $F_2$ are such that
\begin{equation}\label{eqn310}
A_{2,1}=A_{3,1}=A_{2,5}=A_{3,5}=0.
\end{equation}
Then, for any $f \in M\ominus z^k M$, the Fourier functions $f_1$
and $f_2$ must satisfy
\begin{equation}\label{eqn311}
\hat{f_1}(2)= \hat{f_2}(2)=0.
\end{equation}
\end{lem}

\begin{proof}
Consider the case $s=2$ in Lemma \ref{Lemma3}, and suppose that
\eqref{eqn310} holds. The equations in Lemma \ref{Lemma3} will be
satisfied if and only if the following linear system is satisfied:
\[\begin{pmatrix}
\|z^{2k} F_1\|^2_{\alpha} & \langle z^{2k} F_2, z^{2k} F_1 \rangle \\
\langle z^{2k} F_1, z^{2k} F_2 \rangle & \|z^{2k} F_2\|^2_{\alpha}
\end{pmatrix} \cdot \begin{pmatrix} \hat{f_1}(2) \\
\hat{f_2}(2)\end{pmatrix} = \begin{pmatrix} 0 \\ 0\end{pmatrix}.\]
By the Cauchy-Schwarz inequality, the determinant of this system is
strictly positive ($F_1$ and $F_2$ can't be multiples of each other,
because of the presence of $a_4$ and $b_5$). Therefore, the only
solution is the trivial one \eqref{eqn311}.
\end{proof}

We are finally ready to give insight into the role of the recurrence
relations in describing $M \ominus z^kM$:
\begin{lem}\label{Lemma6}
Suppose $F_1$ and $F_2$ satisfy \eqref{eqn310} in Lemma
\ref{Lemma5}, and moreover suppose that
\begin{equation}\label{eqn312}
A_{1,1}=0.\end{equation} Then $M \ominus z^kM$ is spanned by $F_2$
and $F_3$ where
\[
F_3(z)= F_1(z) + \frac{z^k A_{1,5}}{|A_{1,2}|^2 - A_{1,3}A_{1,4}}
\left(A_{1,3}F_2(z) - \overline{A_{1,2}} F_1(z)\right).\]
\end{lem}

Before proceeding with the proof, notice that the denominator of the
second term of the right-hand side is different from 0 by
Cauchy-Schwarz inequality, and so $F_3$ is well defined. Also,
notice that we could also say that $M \ominus z^k M$ is generated by
$F_2$ and $F_4$ where \[F_4(z)= F_1(z) \left( 1- \frac{z^k
A_{1,5}\overline{A_{1,2}}}{|A_{1,2}|^2 - A_{1,3}A_{1,4}}\right).\]

\begin{proof}
Let $f \in M \ominus z^k M$. Since the hypothesis of Lemma
\ref{Lemma5} is met, we can assume the Fourier coefficients of order
2 of the function $f$ satisfy \eqref{eqn311}.

Now, we turn to the case $s=1$ in Lemma \ref{Lemma5}, and we make
use of \eqref{eqn311}, eliminating the terms on $\hat{f_1}(2)$ or
$\hat{f_2}(2)$. We then obtain a linear system which can be written
as
\[\begin{pmatrix}
A_{1,3} & \overline{A_{1,2}} \\
A_{1,2} & A_{1,4}
\end{pmatrix} \cdot \begin{pmatrix} \hat{f_1}(1) \\
\hat{f_2}(1)\end{pmatrix} = \begin{pmatrix} 0 \\ -
\hat{f_1(0)}A_{1,5}\end{pmatrix}.\] By Cauchy-Schwartz inequality,
the system is invertible. This means that for each value of
$\hat{f_1}(0)$ there is a unique value of $\hat{f_1}(1)$ and
$\hat{f_2}(1)$ while $\hat{f_2}(0)$ is free. We can then see that
any element in $M \ominus z^k M$ has a set of coefficients of order
0, 1 and 2 that is a linear combination of those for $F_2$ and those
for $F_3$, since these ones arise as the solution when taking
$\hat{f_2}(0) \neq 0 = \hat{f_1}(0)$ (for $F_2$) and, respectively
for $F_3$, $\hat{f_2}(0)=0\neq \hat{f_1}(0)$. Moreover, $F_2$ and
$F_3$ satisfy all the cases $s \geq 3$ of Lemma \ref{Lemma3}
trivially, which implies that $F_2$ and $F_3$ are elements of
$M\ominus z^kM$. Therefore, any other element $G$ of $M\ominus z^k
M$ must be of the form $G = \lambda_0 F_2 + \lambda_1 F_3 + F_5$,
for some constants $\lambda_0, \lambda_1 \in \C$, where the Fourier
coefficients of $F_5$ are non-zero only for order $s \geq 3$.
Bearing in mind that $F_2$ and $F_3$ are in $M \ominus z^k M$, we
can see that if $G \in M \ominus z^k M$, then, also $F_5$ must
belong to $M \ominus z^kM$, but at the same time, $F_5$ must be in
$z^k M$. Hence, $F_5$ is orthogonal to itself, i.e., it is
identically 0. We have seen that indeed, $F_2$ and $F_3$ span all
the elements of $M \ominus z^k M$.
\end{proof}

It could very well be that $A_{1,5}A_{1,2}=0$ and then $F_3 = F_1$
(from the previous Lemma). In that case, $M$ would definitely have
the wandering property. A crucial step will be to show that under
certain assumptions (that obviously will have to imply $A_{1,5}
A_{1,2} \neq 0$) $F_1$ is not generated as a combination of $F_2$
and $F_3$. This may be regarded as the main point in the proof:

\begin{lem}\label{thm7}
Suppose $F_1$ and $F_2$ satisfy \eqref{eqn310} in Lemma
\ref{Lemma5}, and \eqref{eqn312} in Lemma \ref{Lemma6}.
Additionally, suppose that
\begin{equation}\label{eqn315}
A_{1,5}A_{1,2} \neq 0
\end{equation}
and
\begin{equation}\label{eqn316}
A_{1,3}A_{1,4}-|A_{1,2}|^2 < |A_{1,5}A_{1,2}|.
\end{equation}
Then $F_1 \notin \{M \ominus z^k M\}_{z^k}$. In particular, $M$ does
not have the wandering property.
\end{lem}

Before we proceed with the proof, some remarks are in order:
\begin{rem}\label{rem8}
\begin{itemize}
\item[(1)] Since the left-hand side of \eqref{eqn316} is strictly
positive (once more due to Cauchy-Schwartz inequality), the property
\eqref{eqn316} clearly implies \eqref{eqn315}. However, this
assumption made explicit justifies studying the quotient of both
quantities \[\frac{A_{1,3}A_{1,4}-|A_{1,2}|^2}{ |A_{1,5}A_{1,2}|}.\]
\item[(2)] It seems relevant to interpret condition \eqref{eqn316}
in terms of Cauchy-Schwartz inequality. This condition tells us then
that $F_1$ and $F_2$ are \emph{almost parallel}. However, the
remaining conditions \eqref{eqn310} and \eqref{eqn312} play against
this as an obstacle.
\item[(3)] Conditions \eqref{eqn310}, \eqref{eqn312} and
\eqref{eqn315} do not depend at all on the choice of $a_4$ and
$b_5$, although in order for the existence of the Fourier
expressions we needed them to be non-null. This, together with the
fact that \eqref{eqn316} is an open condition (a strict inequality
between 2 quantities that are continuous as functions of $a_4$ and
$b_5$, seen as variables), simplifies the problem: if we can find a
solution to \eqref{eqn310}, \eqref{eqn312}, \eqref{eqn315} and
\eqref{eqn316} with $a_4=b_5=0$, then a small enough perturbation
changing only the values of these 2 coefficients, will still meet
all the conditions.
\end{itemize}
\end{rem}

Now we are going to show the proof of Lemma \ref{thm7}.

\begin{proof}
Denote \[c = \frac{A_{1,3}A_{1,4}-|A_{1,2}|^2}{ |A_{1,5}A_{1,2}|}.\]
 $c$ is well defined because of \eqref{eqn315}. From Lemma \ref{Lemma6}, $F_2$ and $F_3$ span $M \ominus z^k M$, and thus we can
infer that if $F_1 \in \{M \ominus z^k M\}_{z^k}$, then $F_1$ can be
arbitrarily well approximated in the $D_\alpha$ norm by functions of
the form
\[\left(\sum_{i=0}^t \lambda_{i,t} (z^k + c) z^{ki} \right) F_1(z) +
f_2(z^k) F_2(z),\] where $f_2 \in D_\alpha$ and $\lambda_{i,s}$ are
constants for each $i$ and $s$. Since the terms of order $ks+4$ ($s
\geq 1$) of $F_1$ are all null, this means that $\left(\sum_{i=0}^t
z^{ki} \lambda_{i,t} \right) (z^k + c) F_1(z)$ must approximate
$F_1$ in $D_\alpha$-norm arbitrarily well. Now we will use rather
basic properties of (classical) shift-invariant subspaces of the
spaces $D_\alpha$. The preliminaries for these may be found in
\cite{BS84} and \cite{EFKMR}. Firstly, $(z^k+c)F_1(z)$ generates the
same $z-$invariant subspace of $D_\alpha$ as $F_1$. This implies
that $z^k +c$ is a cyclic function for the shift operator acting on
$D_\alpha$, which in turn, implies that $z^k +c$ must have no zeros
in $\D$. This is only possible if $|c| \geq 1$, which contradicts
the assumption \eqref{eqn316}, completing the proof.
\end{proof}

\begin{rem}
Although for $\alpha > 1$ the cyclicity of $z^k +c$ would imply the
stronger condition that $|c|>1$, when including $a_4$ and $b_5$ in
the computations we need to allow for additional space, so the
condition that we are using seems sharp from this point of view in
all spaces.\end{rem}

\begin{rem}\label{finalsufficient}
What remains, in order to complete the proof of Theorem \ref{main},
is to find a suitable choice of $a_0,...,a_3, a_k,...,a_{k+3},
b_0,...,b_3$ satisfying all the conditions \eqref{eqn310},
\eqref{eqn312}, \eqref{eqn315} and \eqref{eqn316}. \end{rem}

\section{A particular choice of parameters meeting the sufficient conditions}\label{promain2}

In the previous Section we have identified the key optimization
problem of determining whether
\begin{equation}\label{eqn400}
\inf \frac{A_{1,3}A_{1,4}-|A_{1,2}|^2}{ |A_{1,5}A_{1,2}|} <1,
\end{equation}
where the infimum runs through all nonzero vectors
\[(a_0,...,a_3,a_k,...,a_{k+3},b_0,...,b_3) \in \C^{12}\] for which
the equations \eqref{eqn310}, \eqref{eqn312}, \eqref{eqn315} hold.
If this is the case, then the space $M$ will fail to have the
wandering property ($M$ is described in \eqref{eqn300}, where $F_1$
and $F_2$ are as in \eqref{defF1} and \eqref{defF2}, the remaining
notation is in Definition \ref{notation}).

Since $A_{1,2}, A_{1,3}, A_{1,4}, A_{1,5}$, are all 2-homogeneous,
the objective function in \eqref{eqn400} is 0-homogeneous. Hence, we
could in principle choose, for example,
\begin{equation}\label{eqn401} A_{1,5}A_{1,2}=1 \end{equation} and minimize the
numerator. Even though the road we will follow will be different, we
think that this expression in terms of an optimization problem is of
interest in itself. In this Section the main objective will
therefore be to show the following result:

\begin{thm}\label{CSIthm}
For $\alpha = -16$ and $k=6$,
\[\inf A_{1,3}A_{1,4}-|A_{1,2}|^2 < 1\]
where the infimum runs through all the non-null possible vectors of
coefficients \[(a_0,...,a_3,a_k,...,a_{k+3},b_0,...,b_3) \in
\C^{12}\] such that \eqref{eqn310}, \eqref{eqn312}, \eqref{eqn315}
and \eqref{eqn401} hold.
\end{thm}

Denote by $B_0$ the function $\C^{12} \rightarrow \R$ given by
\begin{equation}\label{eqn440}B_0
=\frac{A_{1,3}A_{1,4}-|A_{1,2}|^2}{|A_{1,2}A_{1,5}|}.\end{equation}
We will show the validity of \eqref{eqn400}, which is equivalent to
Theorem \ref{CSIthm}. We start with several further reductions of
the problem that are still applicable in big generality.

We will later change twice the objective function that we will
optimize, but bounds for $B_0$ will follow from bounds for those
other functions $B_1$ and $B_2$. We will find an explicit bound for
$B_2$, while $B_1$ will reappear on Section \ref{numerics} on
numerical results.

\subsection{General method}

The conditions \eqref{eqn310} and \eqref{eqn312}, which are required
for vectors $(a_0,...,a_3,a_k,...,a_{k+3},b_0,...,b_3)$ to be
eligible to the minimization problem in \eqref{eqn400}, may be
expressed altogether in a pair of matrix systems as follows:
\begin{equation}\label{eqn430}
 N \begin{pmatrix}
a_0 \overline{a_k} \\
a_1 \overline{a_{k+1}} \\
a_2 \overline{a_{k+2}} \\
a_3 \overline{a_{k+3}}
\end{pmatrix} = \begin{pmatrix}
0 \\
0 \\
0\end{pmatrix}, \quad  N \begin{pmatrix}
b_0 \overline{a_k} \\
b_1 \overline{a_{k+1}} \\
b_2 \overline{a_{k+2}} \\
b_3 \overline{a_{k+3}}
\end{pmatrix} = \begin{pmatrix}
\overline{A_{1,5}} \\
0 \\
0\end{pmatrix},\end{equation} where $N$ is the $3 \times 4$ matrix
given by
\begin{equation}\label{eqn431}
N = \begin{pmatrix}
\omega_k & \omega_{k+1} & \omega_{k+2} & \omega_{k+3} \\
\omega_{2k} & \omega_{2k+1} & \omega_{2k+2} & \omega_{2k+3} \\
\omega_{3k} & \omega_{3k+1} & \omega_{3k+2} & \omega_{3k+3}
\end{pmatrix}.\end{equation}

\begin{rem}
Whenever the ambient space is either the Hardy space or the
Dirichlet space, the above system is incompatible with the equation
\eqref{eqn315}, since the last row is a linear combination of the
previous ones. This problem only appears with the \emph{usual norm}
in those spaces, given by $\omega_k = 1$ or $\omega_k=k+1$,
respectively, and linear combinations of these.
\end{rem}

We also introduce the following notation which will simplify the
linear expressions:
\begin{equation}\label{eqn432}
N_0 = \begin{pmatrix}
1 & 0 & 0 & 0 \\
\omega_k & \omega_{k+1} & \omega_{k+2} & \omega_{k+3} \\
\omega_{2k} & \omega_{2k+1} & \omega_{2k+2} & \omega_{2k+3} \\
\omega_{3k} & \omega_{3k+1} & \omega_{3k+2} & \omega_{3k+3}
\end{pmatrix},\end{equation}
and
\begin{equation}\label{eqn433}
N_1 = \begin{pmatrix}
\omega_{k+1} & \omega_{k+2} & \omega_{k+3} \\
\omega_{2k+1} & \omega_{2k+2} & \omega_{2k+3} \\
\omega_{3k+1} & \omega_{3k+2} & \omega_{3k+3}
\end{pmatrix}.\end{equation}
By introducing two variables $Z_1$ and $Z_2$, that will play a mute
role, we may convert the system \eqref{eqn430} into a square linear
system with the matrix $N_0$ instead of $N$:
\begin{equation}\label{eqn434}
 N_0 \begin{pmatrix}
a_0 \overline{a_k} \\
a_1 \overline{a_{k+1}} \\
a_2 \overline{a_{k+2}} \\
a_3 \overline{a_{k+3}}
\end{pmatrix} = \begin{pmatrix}
\overline{Z_1} \\
0 \\
0 \\
0\end{pmatrix}, \quad  N_0 \begin{pmatrix}
b_0 \overline{a_k} \\
b_1 \overline{a_{k+1}} \\
b_2 \overline{a_{k+2}} \\
b_3 \overline{a_{k+3}}
\end{pmatrix} = \begin{pmatrix}
\overline{Z_2} \\
\overline{A_{1,5}} \\
0 \\
0\end{pmatrix}.\end{equation} Notice $N_0$ is invertible whenever
$N_1$ is. This should be checked in each case: for us, the values
corresponding to $\alpha=-16$ and $k=6$ give a determinant of $N_1$
that is not 0 and so $N_0$ is invertible (for spaces other than
Hardy or Dirichlet, invertibility will typically hold for all values
of $k$ except perhaps a finite number of cases). If $Z_1$ and $Z_2$
are non-null, then so are $b_0,b_1,b_2$, and $b_3$. Hence, we can
express the equations in terms of the elements of $N_0^{-1}$. Denote
the elements of the first two rows of $N_0^{-1}$ by
\begin{equation}\label{eqn435}
\begin{pmatrix}
1 & 0 \\
E_1 & G_1 \\
E_2 & G_2 \\
E_3 & G_3
\end{pmatrix},\end{equation}
and denote \[ Z_3:= x+iy := \frac{Z_2}{A_{1,5}},\] which is
well-defined if we assume \eqref{eqn315} ($Z_3 \in \C, x,y \in \R$).
Since $\omega=\{\omega_k\}_{k \in \N} \subset \R$, then all of $E_1,
E_2, E_3, G_1, G_2$ and $G_3$ are also real numbers (and for each
fixed space and $k$ they are a specific and computable set of
numbers). Suppose that $E_i \neq 0$ for $i=1,2,3$. Then the system
\eqref{eqn434} is equivalent to the following one, which gives the
values of all the parameters $b_0,...,b_3$ and $a_k,...,a_{k+3}$ in
terms of those of $a_0,...,a_3$:
\[\begin{pmatrix}
a_k \\
a_{k+1} \\
a_{k+2} \\
a_{k+3}
\end{pmatrix} = \begin{pmatrix}
\frac{Z_1}{\overline{a_0}} \\
\frac{E_1 Z_1}{\overline{a_1}} \\
\frac{E_2 Z_1}{\overline{a_2}} \\
\frac{E_3 Z_1}{\overline{a_3}}
\end{pmatrix}, \quad  \begin{pmatrix}
b_0 \\
b_1 \\
b_2 \\
b_3
\end{pmatrix} = \begin{pmatrix}
a_0 \frac{\overline{A_{1,5}Z_3}}{\overline{Z_1}} \\
a_1 \frac{\overline{A_{1,5}}}{\overline{Z_1}} (\overline{Z_3} + \frac{G_1}{E_1}) \\
a_2 \frac{\overline{A_{1,5}}}{\overline{Z_1}} (\overline{Z_3} + \frac{G_2}{E_2}) \\
a_3 \frac{\overline{A_{1,5}}}{\overline{Z_1}} (\overline{Z_3} +
\frac{G_3}{E_3})
\end{pmatrix}.\] In this case, any choice of $Z_3 \in \C$ and $a_0,...,a_3, Z_1,
A_{1,5} \in \C \backslash \{0\}$ will yield the value of the
remaining variables so that \eqref{eqn310} and \eqref{eqn312} are
met. We can express the \emph{objective function} $B_0$ (as in
\eqref{eqn440}) in terms of these 7 variables only, and then find a
good choice of those 7 variables to provide a low value of $B_0$.

To do so, we translate each of the significant quantities for $B_0$,
that is $A_{1,3}, A_{1,4}$ and $A_{1,2}$, in terms of our variables.
In this way, $A_{1,3}$ is given by:
\[
A_{1,3}= |a_0|^2 \omega_k  + |a_1|^2 \omega_{k+1}   + |a_2|^2
\omega_{k+2} + |a_3|^2 \omega_{k+3} + \]\[ + |Z_1|^2 \left(
|a_0|^{-2} \omega_{2k} + |a_1|^{-2} E_1^2 \omega_{2k+1}  +
|a_2|^{-2} E_2^2 \omega_{2k+2} + |a_3|^{-2} E_3^2
\omega_{2k+3}\right).\] We introduce the following notations:
\[d_i = |a_i|^2, \quad i=0,1,2,3.\]
\[H_0 := 1, \quad H_i := E_i^2, \quad i=1,2,3.\]
\[C_1 = \sum_{i=0}^3 d_i \omega_{k+i},\quad C_2 = \sum_{i=0}^3 \frac{H_i \omega_{2k+i}}{d_i}.\]

This notation will be convenient because $C_1$ and $C_2$ only
depend, for a fixed space and $k$, on the choice of $d_i$ (which can
be taken to be any positive real number). By making this choice, the
expression for $A_{1,3}$ becomes much simpler: \[A_{1,3}= C_1
+|Z_1|^2 C_2.\] On the other hand, if we make use of the same
notation for $A_{1,4}$, and defining
\[D_i = -\frac{G_i}{E_i}, \quad i=1,2,3,\]
we may see that
\[A_{1,4}= \frac{|A_{1,5}|^2}{|Z_1|^2} \left( d_0 |Z_3|^2 \omega_k + \sum_{i=1}^3 d_i |Z_3-D_i|^2
\omega_{k+i}\right).\] To follow with previous notation, we give
names of the form $C_t$ to quantities that depend only on
$d_0,...,d_3$ and on the parameters arising from the choice of space
and the value of $k$:
\[C_3 = 2 \sum_{i=1}^3 D_i d_i \omega_{k+i}, \quad _4 = \sum_{i=1}^3 D_i^2 d_i \omega_{k+i}.\] These quantities can be
used to simplify the formula for $A_{1,4}$:
\[A_{1,4}=\frac{|A_{1,5}|^2}{|Z_1|^2} \left( C_1 |Z_3|^2 - C_3 x +
C_4 \right).\] Finally, $|A_{1,2}|^2$ is given by \[|A_{1,2}|^2=
\frac{|A_{1,5}|^2}{|Z_1|^2} \left| C_1 Z_3 - \frac{C_3}{2}
\right|^2.\] Observe that the value of $|A_{1,5}A_{1,2}|$ can also
be obtained now as a function of $A_{1,5}$ and the other free values
$d_0,...,d_3$, $Z_3$ and $Z_1$. Now we are ready to give the desired
form of $B_0$: \[B_0 = \frac{ |Z_1|^2 C_2 (C_1 |Z_3|^2 - C_3 x +
C_4) + C_1 C_4 -\frac{C_3^2}{4}}{|Z_1|\left| C_1 Z_3 - \frac{C_3}{2}
\right| } .\] The first conclusion from here is that the arguments
of $a_i$ and $Z_1$ do not play any role so we can work directly with
the positive real variables $d_0,d_1,d_2,d_3$ and assume $Z_1 \in
\R^+$. Since $H_i$ is always positive, so are $C_1, C_2$ and $C_4$.

Moreover, $B_0$ is positive: This is clear from its definition on
equation \eqref{eqn440}, if we bear in mind Definition
\ref{notation} and Cauchy-Schwarz inequality. As a function of $Z_1$
(considering other variables as constants), $B_0$ takes the form:
\[B_0= \frac{e_0}{B_1}+e_1B_1,\] where $e_0$ and $e_1$ do
not depend on $Z_1$. Now, notice that this implies that both $e_0$
and $e_1$ are non-negative independently of the choices for
$d_0,...,d_3, Z_3, A_{1,5}$. Otherwise, there would be choices of
$Z_1$ for which $B_0$ is negative, which is a contradiction. It can
be shown that in fact, both $e_0$ and $e_1$ are strictly positive
but in fact, if one of them or both is null, then we can take a
value of $Z_1$ such that $B_0$ is arbitrarily close to 0, and the
Theorem \ref{CSIthm} would hold. Therefore, $B_0=B_0(Z_1)$ is a
differentiable function of $Z_1 \in (0, +\infty)$ with
\[\lim_{Z_1 \rightarrow 0} B_0(Z_1) = \lim_{Z_1 \rightarrow +\infty}
B_0(Z_1) = +\infty,\] and so it must have a global minimum on
$\R^+$. This must be at the only point at which $0 = \frac{\partial
B_0}{\partial Z_1}$ which is at \[Z_1^*= \sqrt{\frac{e_0}{e_1}}.\]
There, we obtain the corresponding value for $B_0$: \[B_0(Z_1^*)= 2
\sqrt{e_0e_1}.\] We have reduced the problem of minimizing $B_0$ on
2 variables since we know which value to choose for $Z_1$ and
$A_{1,5}$ does not play any role anymore on the result (provided it
is not 0). Now it is the moment to study $e_0$ and $e_1$ and find a
good choice for $Z_3$ and for $d_0,...,d_3$. We introduce one more
bit of notation:
\[C_5= C_1C_4-\frac{C_3^2}{4},\]
so that $e_0$ is given by \[e_0= \frac{C_5}{|C_1 Z_3 -
\frac{C_3}{2}|},\] whereas $e_1$, by \[e_1 = \frac{C_2 \left( C_1
|Z_3|^2 - C_3 x + C_4 \right)}{|C_1 Z_3 - \frac{C_3}{2}|}.\]

Our objective is to show that $B_0$ can be made smaller than 1,
which will hold if and only if $B_0^2 <1$ for some values of $Z_3,
d_0,...,d_3$. For this reason, we intend to show that
\[4e_0e_1 < 1.\]
The left-hand side above can now be described only in terms of $C_j$
($j=1,...,5$) and $Z_3$: \[4e_0e_1 = \frac{4C_2C_5}{C_1} \left( 1 +
\frac{C_5}{|C_1 Z_3- \frac{C_3}{2}|^2}\right).\]

$C_5$ is positive as may be seen from the expressions for $C_1$,
$C_3$ and $C_4$. Thus, the existence of a set of values $Z_3,
d_0,...,d_3$ such that $4e_0e_1 < 1$ is equivalent with the
existence of (perhaps different values) $d_0,...,d_3$ so that
$4C_2C_5/C_1 < 1$. To see this, observe that this last condition is
necessary because $C_5$ is positive, and so $4e_0e_1 >
\frac{4C_2C_5}{C_1}$. Also, suppose that for some choice of
$d_0,...,d_3$ and for some $\varepsilon
>0$, we have
\[\frac{4C_2C_5}{C_1} < 1- \varepsilon.\]
Since $C_1, C_2, C_3, C_5$ do not depend on $Z_3$,  we can choose
$Z_3$ large enough in modulus so that \[\frac{C_5}{|C_1 Z_3 -
\frac{C_3^2}{4}|^2} < \varepsilon.\] This yields that
\[4e_0e_1 < 1 - \varepsilon ^2 < 1.\]

We can define the \emph{second objective function}
\begin{equation}\label{eqn601}B_1:=\frac{4C_2C_5}{C_1}\end{equation} (and it is a function of
$d_0,...,d_3$ only). We have reduced our original problem to the
much simpler one of proving
\begin{equation}\label{eqn470}
\inf_{d_0,d_1,d_2,d_3 \in \R^+} B_1 <1.
\end{equation}

Let us remark that $C_1$, $C_2$ and $C_5$ are respectively $1-$,
$(-1)-$ and $2-$ homogeneous, and so $B_1$ is 0-homogeneous.
Therefore, from now on we fix $d_0=1$ and we do not loose anything
(even if the best choice is for $d_0=0$, the condition we want to
show is open and  depends continuously on the parameters).

From the definition of $C_5$ we can see that $B_1 < 4C_2 C_4$. In
the next subsection, we will find values for which
\begin{equation}\label{eqn495}
B_2 (d_1,d_2,d_3) := 4C_2C_4(1,d_1,d_2,d_3)\end{equation} is bounded
above by 1, therefore showing \eqref{eqn470}. Hence we give a name
to $B_2$ as well: we call it the \emph{third objective function}.
Our final objective becomes the following:
\[\inf_{d_1,d_2,d_3 \in \R^+} B_2 <1.\]
As mentioned earlier, we will use again $B_1$ in the numerical
computations section, since it is easier to obtain a value below 1
for a particular set of parameters for $B_1$, but $B_2$ is easier to
deal with in abstract because of its simpler expression. We are left
with a 3 variable optimization problem concerning $d_1$, $d_2$ and
$d_3$ in $\R^+$. To proceed from here, one needs a good
understanding of the quantities $E_1, E_2, E_3, G_1, G_2, G_3$
defining $C_1, C_2$ and $C_5$. Recall the definitions of $N_0$ and
$N_1$ in \eqref{eqn432} and \eqref{eqn433}. Through \eqref{eqn435},
we defined $E_1, E_2, E_3$ as the
solution to \[N_0 \begin{pmatrix} 1 \\ E_1 \\
E_2
\\ E_3 \end{pmatrix} = \begin{pmatrix} 1 \\ 0 \\ 0
\\ 0 \end{pmatrix}.\]
The value 1 can be integrated into the independent term, giving the
equivalent system \begin{equation}\label{eqn472} \begin{pmatrix} E_1 \\
E_2 \\ E_3\end{pmatrix} = - N_1^{-1} \begin{pmatrix} \omega_k
\\ \omega_{2k} \\ \omega_{3k} \end{pmatrix}.\end{equation}

In an analogous way, we can obtain $G_1, G_2, G_3$ as the solution
to
\begin{equation}\label{eqn473} \begin{pmatrix}
G_1 \\ G_2 \\ G_3\end{pmatrix} = N_1^{-1} \begin{pmatrix} 1
\\ 0 \\ 0 \end{pmatrix}.\end{equation}

Cramer's rule tells us how to obtain the values of $E_1, E_2, E_3,
G_1, G_2$ and $G_3$ provided that
\[\det N_1 \neq 0.\]

Equations \eqref{eqn473} and \eqref{eqn472} allow to obtain the
values of
\[ G_1 = \frac{\begin{vmatrix}\omega_{2k+2} & \omega_{2k+3} \\
\omega_{3k+2} & \omega_{3k+3}
\end{vmatrix}}{\det N_1}; \quad E_1 = -\frac{\begin{vmatrix} \omega_k  & \omega_{k+2} & \omega_{k+3} \\ \omega_{2k} & \omega_{2k+2} & \omega_{2k+3} \\
\omega_{3k} & \omega_{3k+2} & \omega_{3k+3}
\end{vmatrix}}{\det N_1}; \]
\[ G_2 = \frac{\begin{vmatrix}\omega_{2k+3} & \omega_{2k+1} \\
\omega_{3k+3} & \omega_{3k+1}
\end{vmatrix}}{\det N_1}; \quad E_2 = -\frac{\begin{vmatrix} \omega_{k+1}  & \omega_{k} & \omega_{k+3} \\ \omega_{2k+1} & \omega_{2k} & \omega_{2k+3} \\
\omega_{3k+1} & \omega_{3k} & \omega_{3k+3}
\end{vmatrix}}{\det N_1};  \]
\[ G_3 = \frac{\begin{vmatrix}\omega_{2k+1} & \omega_{2k+2} \\
\omega_{3k+1} & \omega_{3k+2}
\end{vmatrix}}{\det N_1}; \quad E_3 = -\frac{\begin{vmatrix} \omega_{k+1}  & \omega_{k+2} & \omega_{k} \\ \omega_{2k+1} & \omega_{2k+2} & \omega_{2k} \\
\omega_{3k+1} & \omega_{3k+2} & \omega_{3k}
\end{vmatrix}}{\det N_1}.
\]

We need to compute these 7 determinants to go further and for this
we need to particularize on the values of $k$ and $\alpha$ of our
choice. Notice, however, that the 12 numbers determining the matrix
$N$ in \eqref{eqn431} also determine the answer to the optimization
problems presented.

The numerical observations of Section \ref{numerics} will suggest
that for each $k \geq 6$ there exists a small enough $\alpha_0$ such
that for all $\alpha \leq \alpha_0$, some $z^k$-invariant subspace
$M \subset D_\alpha$ fails to have the wandering property. The
numerical construction will be a minor modification of the one
presented up to this point. We prove that this is indeed the case
for $k=6$ and $\alpha= -16$ in the next subsection.

\subsection{Values for $k=6$ and $\alpha = -16$}

At this point we choose $k$ equal to 6 and $\omega=\{\omega_t\}_{t
\in \N}$ to be defined by \[\omega_t= (t+1)^{-16}.\] Since the scale
of all the numbers appearing in the matrix $N$ is dominated by
$7^{-16}$, it is numerically useful to express the elements of $N$
as $7^{-16}$ times other numbers. Going forward, for each figure we
write the decimal numbers that are the correct rounding number to
the nearest decimal. When the precision to which a number is
determined is below this level, this is expressed. In this way, we
obtain the values of $\omega_t$ indicated in Table \ref{table3},
which are given also compared to $\omega_k = 7^{-16} = 3.0090635
\cdot 10^{-14}$.
\begin{table}\label{table3}
\caption{Values of $\omega_t$ in the matrix $N$} \centering
\begin{tabular}{c c c}
\hline \hline $t$ & $\omega_t$ & $\omega_t \cdot 7^{16}$ \\
[0.5ex] \hline
 $k   $  &  $7^{-16} $  &   $1                          $ \\
 $k+1 $  &  $8^{-16} $  &   $1.1806708702 \cdot 10^{-1} $ \\
 $k+2 $  &  $9^{-16} $  &   $1.793446761 \cdot 10^{-2}  $ \\
 $k+3 $  &  $10^{-16}$  &   $3.32329305 \cdot 10^{-3}   $ \\
 $2k  $  &  $13^{-16}$  &   $4.99430433671 \cdot 10^{-5}$ \\
 $2k+1$  &  $14^{-16}$  &   $1.52587890625 \cdot 10^{-5}$ \\
 $2k+2$  &  $15^{-16}$  &   $5.05951042777 \cdot 10^{-6}$ \\
 $2k+3$  &  $16^{-16}$  &   $1.80156077608 \cdot 10^{-6}$ \\
 $3k  $  &  $19^{-16}$  &   $1.15215530802 \cdot 10^{-7}$ \\
 $3k+1$  &  $20^{-16}$  &   $5.0709427749 \cdot 10^{-8} $ \\
 $3k+2$  &  $21^{-16}$  &   $2.32305731254 \cdot 10^{-8}$ \\
 $3k+3$  &  $22^{-16}$  &   $1.10358489374 \cdot 10^{-8}$ \\
\hline
\end{tabular}
\end{table}

For each of the determinants to be computed precisely, it is
convenient to use the multiplicative properties of determinants to
extract the factors of the form $7^{-16}$, as well as powers of 10
in order to make the numbers in the determinants, mesoscopic. This
will yield

\[\det N_1 = 7 ^{-48} \cdot 10^{-16} \cdot
\begin{vmatrix}11.8... &  1.793... & 0.3323 \\ 15.25... & 5.0595... & 1.8015... \\ 5.07... & 2.323...
 & 1.1035...\end{vmatrix}.\] From here on, all the bounds are obtained from classical interval arithmetic. The above implies
\[ \det N_1 = (1.6207616 \pm 5\cdot 10^{-8}) \cdot 10^{-56}.\]

With this same philosophy we may obtain that \begin{eqnarray*}G_1 =
(7.8126227 \pm 6 \cdot 10^{-7}) \cdot 10^{14}, \quad E_1= -16.37478
\pm 2 \cdot
10^{-6},\\
G_2 = (-4.3037417 \pm 3 \cdot 10^{-7}) \cdot 10^{15}, \quad E_2= 65.63437 \pm 7 \cdot 10^{-6},\\
G_3 = (5.4695405 \pm 4 \cdot 10^{-7}) \cdot 10^{15}, \quad E_3=
-73.35945 \pm 2 \cdot 10^{-5}.\end{eqnarray*}

These results give upper bounds for the values appearing in the
definitions of $C_2$ and $C_4$: $H_i$ and $D_i^2$. In particular, we
may see that \begin{eqnarray*}H_1 = E_1^2 \leq 268.13349, \quad
D_1^2 = \frac{G_1^2}{E_1^2} \leq 2.276371
\cdot 10^{27},  \\
H_2 = E_2^2 \leq 4307.8715, \quad D_2^2 = \frac{G_2^2}{E_2^2} \leq
4.29962 \cdot 10^{27}, \\
H_3 = E_3^2 \leq 5381.61, \quad D_3^2 = \frac{G_3^2}{E_3^2} \leq
5.55892 \cdot 10^{27}.\end{eqnarray*}

Now we have all elements of the third objective function defined on
\eqref{eqn495}. We can obtain the bounds \[C_2 \leq 13^{-16} +
\frac{14^{-16} \cdot 268.13349}{d_1} + \frac{15^{-16} \cdot
4307.8715}{d_2} + \frac{16^{-16} \cdot 5381.61}{d_3},\] which can be
translated into
\begin{equation}\label{eqn496}
C_2 \leq 3.01 \cdot 10^{-20} \cdot \left(49.944 +
\frac{4091.393}{d_1} + \frac{21795.721}{d_2} +
\frac{9695.298}{d_3}\right).
\end{equation}
For $C_4$, the bound we obtain is
\begin{equation}\label{eqn497}
C_4 \leq 3.01 \cdot 10^{11} \cdot \left(26.877 d_1 + 7.712 d_2 +
1.848 d_3 \right).
\end{equation}

Putting together \eqref{eqn496} and \eqref{eqn497} provides
estimates for the third objective function in \eqref{eqn495}. At
this point, we could find the best possible values for $d_i$ but if
we make an educated guess and choose $d_1= 1$, $d_2=4$ and $d_3 = 6$
we will already obtain a good enough estimate:
\begin{equation}\label{eqn498}
B_4(1,4,6) \leq 0.02795
\end{equation}
This concludes the proof of the Theorem \ref{CSIthm}.

\subsection{Extrapolation to an interval}\label{interval}

What remains to complete the proof of Theorem \ref{main} is to
extend the solution to a small interval around $\alpha = -16$.
Notice that in the proof of Theorem \ref{CSIthm}, every bound that
we need depends only on quotients of minors of the matrix $N$. All
these determinants vary continuously with the values of the sequence
$\omega$, and since they do not become 0 when $\alpha=-16$, all the
bounds depend continuously on the value of $\alpha$ around
$\alpha=-16$. On the other hand, following the same path described
in the proof for a different choice of $\omega$ will also lead to a
solution to the linear equations in the system \eqref{eqn430}, that
determine which values of parameters we can try in the minimization
problems. It is only the values of $E_i$ and $G_i$ (and hence those
of their derived quantities, $H_i$, $D_i$ and $C_i$) that will vary,
affecting the minimization problem.

Because of this, the optimal value for the minimization problems
will depend continuously on each value $\omega_t$, for $t \in \N$
(as well as on the choice of $d_1$, $d_2$ and $d_3$). This means
that there must exist an interval around $\alpha=-16$ for which
$B_2$ stays below 1. Theorem \ref{main} is now completely proved.

\subsection{Explicit recovery of the parameters}

It is our intention now to recover explicitly the numerical values
of the parameters intervening in the previous subsections, leading
to the proof of Theorem \ref{main}, although the latter has already
been established. We do this for completion, and also, in order to
pave the path for a blind but very simple and explicit proof of the
Theorem \ref{CSIthm}. Such proof consists on simply plugging the
obtained values for $a_i$ and $b_i$ into the equations
\eqref{eqn430} and then check that the corresponding value of $B_0$
is less than 1.

At the end of the previous subsection we chose the values $d_1=1$,
$d_2=4$ and $d_3=6$ which implies that we can take $a_1=1$, $a_2=2$
and $a_3= \sqrt{6}$ to solve the optimization problem in
\eqref{eqn400} (recall $a_i$ were taken positive real and $d_i$ were
their squares). We had also previously chosen $d_0=1$, and hence,
$a_0=1$. The objective function was bounded in \eqref{eqn498} by
$4C_2C_5/C_1 < 1 - \varepsilon = 0.02795,$ which yields a value of
$\varepsilon = 0.97205$. From these values we can bootstrap and
recover each number that was needed in the proof, explicitly. We
give the results we obtain in Table \ref{table4}. Each number on
that table can be found before those other numbers below it. The
choice for $Z_3$ is arbitrary within some region, and recall that
$D_i \in \R^{+}$. The value of $A_{1,5}$ can be obtained if we
assume \eqref{eqn401}, from the fact that \[|A_{1,2}|^2 =
\frac{|A_{1,5}|^2}{|Z_1|^2} |C_1Z_3-C_3/2|^2 =
\frac{1}{|A_{1,5}|^2}.\] One can then, solve to obtain $A_{1,5}$ as
a function of the already known parameters.

\begin{table}\label{table4}
\caption{Values of the parameters in the proof} \centering
\begin{tabular}{c c c}
\hline \hline Notation & Approximate value & Upper bound \\
[0.5ex] \hline
 $C_4              $  &  $2.07 \cdot 10^{13}       $ &  \\
 $C_2              $  &  $3.372 \cdot 10^{-16}     $ &  \\
 $C_1              $  &  $3.379494 \cdot 10^{-14}  $ &  \\
 $C_3              $  &  $0.355785                 $ &  \\
 $Z_3              $  &  $-2 \cdot 10^{13}         $ &  \\
 $|C_1 Z_3 -C_3/2| $  &  $1.03168                  $ &  \\
 $C_5              $  &  $0.66791                  $ &  \\
 $\frac{C_5}{|C_1Z_3-\frac{C_3}{2}|^2}$  &           &   $0.628$  \\
 $B_0^2$              &                              &   $0.0463$ \\
 $B_0$                &                              &   $0.216$  \\
 $e_0$                &  $0.6474$                    &  \\
 $e_1$                &  $0.01351  $                 &  \\
 $Z_1$                &  $6.92$                      &  \\
 $A_{1,5}$            &  $2.59$                      &  \\
 $a_k$                &  $6.92$                      &  \\
 $a_{k+1}$            &  $-113.3$                    &  \\
 $a_{k+2}$            &  $227.1$                     &  \\
 $a_{k+3}$            &  $-207.2$                    &  \\
 $b_0$                &  $-7.5 \cdot 10^{12}$        &  \\
 $b_1$                &  $-2.53 \cdot 10^{13}$       &  \\
 $b_2$                &  $-1.28 \cdot 10^{14}$       &  \\
 $b_3$                &  $-8.67 \cdot 10^{13}$       &  \\
\hline
\end{tabular}
\end{table}

These are some values that will solve all the equations
\eqref{eqn310}, \eqref{eqn312}, \eqref{eqn315} and \eqref{eqn316}.
One could also want to find an explicit choice of $a_4$ and $b_5$
that makes everything work as described in Remark \ref{rem8}(3). It
is clear that only $A_{1,3}$ and $A_{1,4}$ will be affected by a
permutation in those coefficients (as compared to $A_{1,1},A_{1,2}$
and $A_{1,5}$). One can check that choosing $a_4=b_5 =1$ will not
affect that much the value of $A_{1,3}$ and $A_{1,4}$ and will in
fact be enough to keep satisfying \eqref{eqn316} (notice that in the
expression of the norm $A_{1,3}$, very small weight $11^{-16}$ is
assigned to the coefficient $|a_4|^2$, as well as to $A_{1,4}$
respectively with $12^{-16}$ and $|b_5|^2$).

\begin{rem}
A shorter but blind proof is possible. Once recovered the values of
the parameters as obtained in the previous subsection, it would be
enough, in order to prove Theorem \ref{CSIthm}, to check its
validity for those parameters obtained. Although this may very well
yield a valid proof, we find more illustrative to describe the
method to obtain explicitly those parameters.
\end{rem}

\subsection{Proof of Corollary}\label{pronew}

As mentioned in Section \ref{interval}, the proof of Theorem
\ref{CSIthm} (and thus, that of Theorem \ref{main}) depend only on
properties of the sequence $\omega$: indeed let $\omega \subset
\R^+$ be a sequence with $\omega_0 =1$ and such that
\[\lim_{k \rightarrow \infty} \frac{\omega_k}{\omega_{k+1}} =1.\]
We refer to these as \emph{Hardy-type spaces}. Such spaces include
the $D_\alpha$ spaces, indeed, for the choice $\omega_k =
(k+1)^{\alpha}$ and they have been studied with regards to invariant
subspace properties such as cyclicity for the shift operator in
\cite{FMS}. The proof of the case $\alpha=-16$ in Theorem \ref{main}
extends directly without any relevant changes to any space for which
the corresponding matrix $N$ (that depends on the sequence $\omega$)
is the same as for $D_{-16}$. Notice that modifying the sequence
$\omega$ in a finite number of points will not affect \emph{which
functions form the space}, only the choice of equivalent norm is
affected. Hence, fixed $k \geq 6$, one can modify any space
$D_\alpha$ to an equivalent norm by changing only 12 numbers of the
sequence $\omega$: those appearing in the definition of the matrix
$N$, namely, $\omega_k, ... , \omega_{k+3}, \omega_{2k}, ... ,
\omega_{2k+3}, \omega_{3k}, ... , \omega_{3k+3}$, by the
corresponding ones for $D_{-16}$.

\subsection{Proof for large spaces}\label{pronew2}

We follow the notation from the previous subsections. By choosing
$d_1=d_2=d_3=1$ and fixing $k \geq 9$ we are going to find that the
method described also yields the existence of
$\alpha_k^*=-(5k+\frac{700}{(k-9)^2})$ such that, for all $\alpha
\leq \alpha_k^*$ there exists a $z^k$ invariant subspace in
$D_\alpha$ without the corresponding wandering property, as stated
in Theorem \ref{9k28}. Denote $\beta=|\alpha|$ and we can suppose
$\beta  \geq 1 $. If we compute the determinants giving the values
of $G_1, G_2, G_3, E_1, E_2, E_3$ and denote by $1/d$ the
determinant of $N_1$, we obtain:
\[|G_1|= d ((2k+3)(3k+3))^{\alpha} ((1-\frac{1}{3k+4})^\beta -
(1-\frac{1}{2k+4})^\beta)\] A standard estimate will give
\[ |G_1| \leq d ((2k+3)(3k+3))^{\alpha} \frac{\beta
k}{(2k+4)(3k+4)}.\] In an analogous manner we can estimate all the
other elements $G_i$ to obtain
\[ |G_2| \leq ((2k+4)(3k+4))^{\alpha} \frac{2 d \beta
k}{(2k+2)(3k+2)},\] \[ |G_3| \leq  ((2k+2)(3k+2))^{\alpha} \frac{d
\beta k}{(2k+3)(3k+3)}.\]

Denote $E_0=1$. $E_i$ can also be described in terms of $d$ and
$\omega$ in a similar way by developing the determinants.

For $E_1$, for instance, we can see that
\[|E_1| = d( p_{1,1}(k)^\alpha - p_{2,1}(k)^\alpha + p_{3,1}(k)^\alpha -
p_{4,1}(k)^\alpha + p_{5,1}(k)^\alpha - p_{6,1}(k)^\alpha),\] where
each  $p_{i,1}$ is a polynomial of degree 3 of the form $p_{i,1}(k)
= 6 k^3 + a k^2 + b k + 12$. Since $\beta = -\alpha$ is very large,
we can see that the polynomial making the biggest contribution to
the determinant in the numerator of $E_1$ will be the smallest of
them. In this case, $p_{1,1} = (k+1)(2k+3)(3k+4)= 6 k^3 + 23 k^2 +
29k +12$. The second smallest is $p_{2,1}= (k+1)(2k+4)(3k+3) =  6
k^3 + 24 k^2 + 30k +12$, and we are going to see that
$p_{1,1}^\alpha$ is in fact much larger than $2p_{2,1}^\alpha
> p_{2,1}^\alpha+p_{4,1}^\alpha$, where $p_{4,1}^\alpha > p_{6,1}^\alpha$ is the next biggest. Let $s \in (0,1)$:
\begin{equation}\label{eqn700}
sp_{1,1}^\alpha \geq 2 p_{2,1}^\alpha.\end{equation} Then,
\[p_{1,1}^\alpha > (1-s)p_{1,1}^\alpha + p_{2,1}^\alpha +
p_{4,1}^\alpha.\] The other terms will make a positive contribution,
since $p_{3,1}^\alpha$ is summed to the result and $p_{5,1}^\alpha >
p_{6,1}^\alpha$. Altogether, this gives that
\begin{equation}\label{eqn701}
|E_1| \geq d (1-s) p_{1,1}^\alpha.
\end{equation}
This is in fact, quite sharp, since we always have $|E_1| \leq 3d
p_{1,1}^\alpha$. Let us determine what are the values of $s \in
(0,1)$ for which \eqref{eqn700} holds: We  want
\[\frac{p_{1,1}(k)}{p_{2,1}(k)} \leq (\frac{s}{2})^{1/\beta}.\]
The left hand side is equal to $(1-\frac{k}{(2k+4)(3k+3)})$. Taking
logarithms on both sides and using the standard estimate $\log (x+1)
\leq x$, will give the following sufficient condition for $s$ to
satisfy \eqref{eqn700}:
\begin{equation}\label{eqn702}
\beta \geq 6 \frac{(k+1)}{k} (k+2) \log (2/s).
\end{equation}
The same principle may be applied to $|E_2|$ and $|E_3|$. However,
what polynomials play the role of $p_{1,1}$ and $p_{2,1}$ will
depend on each case. In any case, we will obtain sufficient
conditions analogous to \eqref{eqn702} for $s$ to satisfy an
analogous to \eqref{eqn701}, and it can be checked that these
sufficient conditions are actually less restrictive than
\eqref{eqn702}. This method may also be applied to estimate $d$. The
best upper estimate for $d$ will be
\begin{equation}\label{eqn703}
d \leq \frac{p_{1,*}^\beta}{1-s},\end{equation} where
\[p_{1,*}(k)= (k+2)(2k+3)(3k+4)= 6k^3 + 29k^2 + 46k + 24.\]

A first conclusion is that for all $i=0,1,2,3$ we have the bounds
\begin{equation}\label{eqn704}
(1-s)/3 \leq |E_i| \leq \frac{3}{1-s}
\left(\frac{p_{1,3}(k)}{p_{1,*}(k)}\right)^\alpha.
\end{equation}

Now we are ready to study the objective function $4C_2C_4$.

The bounds
\[C_2 \leq 4 \omega_{2k}  \sup_{i=0,1,2,3} H_i\]
\[C_4 \leq 3 \omega_{k+1} \sup_{i=1,2,3} D_i^2\]
provide a bound for the objective function:
\[4C_2C_4 \leq  48 \omega_{2k} \omega_{k+1} \sup_{i,j=0,1,2,3}
\left(\frac{G_iE_i}{E_j}\right)^2.\] The supremum among $G_i^2$ is
clearly dominated by
\[G^2 := 4 (\beta k d )^2 ((2k+2)(3k+2))^{2\alpha-2} \leq \frac{d^2 \beta^2}{9k^2} \omega_{2k+1}^2 \omega_{3k+1}^2 .\]
From this we can obtain another estimate for the objective function:
\[4C_2C_4 < \frac{48 (\beta \omega_{2k+1} \omega_{3k+1} d)^2
\omega_{2k}\omega_{k+1}}{9 k^2} \sup_{i,j=0,1,2,3}
\left(\frac{E_i}{E_j}\right)^2.\]

Now we may plug in the estimates from above and below for $|E_i|$
obtained in \eqref{eqn704}, which yields
\[4C_2C_4 < \frac{432 \beta^2 \omega_{2k}}{(1-s)^4 k^2 \omega_{k+1}}
\left(\frac{q_2(k)}{q_1(k)}\right)^{2\beta},\] where \[q_1(k)=
(k+1)^4(k+2/3), \quad q_2(k)= (k+2)(k+3/2)^2(k+4/3)^2.\]

This means that the objective function can be bounded by the factor
$\frac{432 \beta^2}{(1-s)^4 k^2} a^\beta$ where $a$ is of the form
\[a=
\left(\frac{1}{2}+\frac{3}{2(2k+1)}\right) \cdot
\left(1+\frac{1}{k+1}\right)^2 \cdot
\left(1+\frac{1}{2(k+1)}\right)^4 \cdot\] \[\cdot
\left(1+\frac{1}{3(k+1)}\right)^2 \cdot
\left(1+\frac{2}{3k+2}\right)^2.
\]
Each of the factors on the right hand side is decreasing on $k$ and
the limit as $k$ increases to $\infty$ is clearly $1/2$, so for some
$k_0 \in \N$ and  all $k \geq k_0$, $a^\beta$ will decay
exponentially fast with $\beta$. This is in fact true for $k \geq
10$, since $a(10)<1$. We have now 2 conditions on $s, k, \beta$ that
if satisfied, guarantee the existence of a $z^k$ invariant subspace
in $D_{-\beta}$ without the wandering property. The conditions
\eqref{eqn702} and that $\frac{432 \beta^2}{(1-s)^4 k^2} a^\beta <
1$.

For any $k \leq 18$, the choices $s=0,985$ and $\beta=5k$ achieves
both things (the bound on the objective function will be decaying on
$k$ and it holds for $k=18$). For $k=10,...,17$ we provide a list of
the corresponding choices for $\beta$ and $s$ that will satisfy both
conditions. Any choice of $\beta$ larger than the one provided will
make the objective function decay as well. The lower bound given for
$\beta$ guarantees \eqref{eqn702}.
\begin{table}\label{table5}
\caption{Solutions for $k \in [10,17]$} \centering
\begin{tabular}{c c c c c}
\hline \hline $k$ & $\beta$ & $s$ & Objective $<$ & $\beta >$
\\ [0.5ex] \hline
 10     &  530  &  0.05  &  0.994   &  293  \\
 11     &  165  &  0.3   &  0.612   &  162  \\
 12     &  120  &  0.6   &  0.387   &  110  \\
 13     &  104  &  0.8   &  0.490   &   89  \\
 14     &  98   &  0.9   &  0.556   &   83  \\
 15     &  90   &  0.93  &  0.562   &   84  \\
 16     &  87   &  0.94  &  0.864   &   87  \\
 17     &  88   &  0.97  &  0.502   &   88  \\
\hline
\end{tabular}
\end{table} See Table \ref{table5}. All the corresponding choices of
$\beta$ are bounded by $5k + \frac{700}{(k-9)^2}$. This concludes
the proof of the Theorem \ref{9k28}.

\section{Other numerical results}\label{numerics}

In the proof of the Theorem \ref{CSIthm}, we have made an assumption
that we have not mentioned anything about: the fact that the degrees
for which $F_1$ and $F_2$ have non-zero coefficients are of the form
$0, ..., 5$ and $k, ..., k+3$. The same proof we followed could a
priori work for any other functions $F_1$ and $F_2$ of the form:

\[F_1(z)= \sum_{i=0}^4 a_i z^{\gamma_i} + \sum_{i=0}^3 a_{k+i} z^{k +
\gamma_i},\] and \[ F_2(z)= \sum_{i=0}^3 b_i z^{\gamma_i} + b_5
z^{\gamma_5},\] provided that $a_i$ and $b_i$ satisfy some equations
similar to \eqref{eqn430} but modified accordingly and that
$\gamma_0,...,\gamma_5$ satisfy, whenever $i \neq j$, that
\[\label{eqn600} \gamma_i \neq \gamma_j \mod k.\]

We ran numerical experiments on the second objective function $B_3$
(as in \eqref{eqn601}) for different values of $\alpha$, $k$, $d_1,
d_2, d_3$ and $\gamma_0, ... , \gamma_3$. We indicate in Table
\ref{table1} the most significant results. Recall that $B_1 < B_2$
and that a value for $B_1$ below 1 implies the existence of a $z^k$
invariant subspace without the wandering property. The values of
$\gamma_i$ have been chosen of the form
\[\gamma_i = \phi_i k + i.\] The values indicated there for several of the
variables were selected in order to guarantee the objective function
is clearly below 1. We keep $d_0=1, \gamma_0=0$,
 and $\gamma_1=1$ in all cases. Hence, we only show the values of $\alpha$, $\phi_2$, $\phi_3$,
 $d_1$, $d_2$, $d_3$ and $B_1$, and a larger value of $B_1$ typically indicates that for that value of $\alpha$ it was more difficult
 to find a suitable example.
Table \ref{table1} only contains values that achieved the objective
for $k=6$.

\begin{table}\label{table1}
\caption{Numerical results for $k=6$} \centering
\begin{tabular}{c c c c c c c}
\hline \hline $\alpha$ & $\phi_2$ & $\phi_3$ & $d_1$ & $d_2$ & $d_3$
& $B_1$ \\ [0.5ex] \hline
 -16     &  0  &  0  &  1   &  4  &       6  & 0.02324  \\
 -16     &  0  &  3  &  1   & 10  &    2000  & 0.00667  \\
 -12     &  1  &  4  &  1   & 20  &    5000  & 0.02397  \\
  -8     &  1  &  8  &  1   & 10  &    5000  & 0.1525   \\
  -7     &  2  & 12  &  1   & 20  &   10000  & 0.31668  \\
  -6     &  3  & 17  &  0.2 & 13  &   16000  & 0.5635   \\
  -5     &  2  & 34  &  4   & 11  &  100000  & 0.99826  \\
  -4.999 &  2  & 34  &  4   & 11  &  100000  & 0.999006 \\
\hline
\end{tabular}
\end{table}

Without varying $k$ from $k=6$ to higher values it was very costly
to determine adequate values for the other parameters for values of
$\alpha > -5$. Allowing for higher values of $k$ made it possible
for the barrier to be moved to $\alpha = -4.2$, as can be seen in
the following Table \ref{table2}.

\begin{table}\label{table2}
\caption{Numerical results for $k>6$} \centering
\begin{tabular}{c c c c c c c c}
\hline \hline $\alpha$ & $\phi_2$ & $\phi_3$ & $d_1$ & $d_2$ & $d_3$
& $k$ & $B_1$ \\ [0.5ex] \hline
 -5    &  2  &  34  &     4   & 11     &   100000  &    7  & 0.875    \\
 -5    &  3  &  35  &     2.2 & 16     &   130000  &   10  & 0.71312  \\
 -4.5  &  3  &  50  &  1000   & 11     &   150000  &   12  & 0.96775  \\
 -4.25 &  3  &  97  &    70   &  0.54  &    70000  &   47  & 0.99436  \\
 -4.22 &  3  & 150  &  5000   &  0.2   &   150000  &   74  & 0.986    \\
 -4.2  &  3  & 166  & 10000   &  0.142 &   150000  &   88  & 0.999
  \\
\hline
\end{tabular}
\end{table}

Since we did not perform validation of the numerics for the
experiments in this Section, the results should be taken not as a
proof but rather as a hint and the values for which $B_1$ is close
to 1 should be taken with special care. However, our results suggest
that increasing $\alpha$ makes the problem harder and there is no
particular reason to expect a much higher threshold than
$\alpha=-4.2$. In any case, several choices have been made in our
strategy of proof that limit the scope of the method. The chosen
values of $\phi_2$, $\phi_3$, $d_1$, $d_2$, $d_3$, and $k$ for each
$\alpha$ were chosen by experimentation with a variational spirit.
We were unable to find any rationale on their behavior.

\section{Further remarks}\label{rema}

\begin{itemize}
\item[(1)] Corollary \ref{coro} shows that the wandering property is
not a property of the ambient space, as a set, but rather of the
choice of equivalent norm. In this way, the Wold decomposition seems
to be an obstruction only when we are dealing with the Hardy space
$H^2$ (and perhaps spaces where the sequence $\omega$ is
increasing).

\item[(2)] It seems that a key feature of the values $k=6$ and
$\alpha=-16$ is that $|\alpha|$ is large enough for $k$. This is
clearly the case in the values of $\alpha <
-(5k+\frac{700}{(k-9)^2})$ when $k\geq 10$, from Theorem \ref{9k28}.
One reasonable direction of research from here is to determine
whether the wandering property in $D_\alpha$ spaces also fails for
$k=1,...,9$. In any case, it seems plausible from all our
observations that for each $k$ there should be a value $\alpha_k^*$,
such that for all $\alpha < \alpha_k^*$, there exist $z^k$ invariant
subspaces in $D_\alpha$ (with their usual norms) without the
wandering property, while for $\alpha > \alpha_k^*$ the wandering
property holds. Perhaps a study of the derivatives of $C_2$ and
$C_4$ in the proof of the Theorem \ref{CSIthm} will give information
about this question. To deepen in this idea, let us comment that, of
the possible spaces of the form $H^2_\omega$, we have seen that many
contain a $z^k$ invariant subspace without the wandering property.
If we consider the set of all eligible sequences $\omega$, this is
true of at least a set of codimension 12. In fact, for each $t \in
\N$ and for each $\omega$ for which there is a solution, variation
of the minimization problem with respect to $\omega_t$ will also be
continuous, so a small interval may be taken around $\omega_t$ that
does not affect the existence of the non-wandering subspace. In
Section \ref{numerics}, even more ways to generate such subspaces
are described. This seems to point in the direction that if
$\omega_t$ \emph{decays fast enough, enough times} there should be
$z^k$-invariant subspaces without the wandering property. The
opposite direction (not fast enough decay or not enough times,
implying wandering property) will be studied in more detail in
\cite{GGPS}.

\item[(3)] $F_1$ and $F_2$ were chosen as polynomials with a very
particular restriction on their coefficients (in Theorem \ref{main},
$F_2$ is of degree 5, and $F_1$ of degree $k+3$ with several more
constraints). It seems plausible to find other arrangements where,
for instance, $F_2$ is of the form \[F_2(z)= b_k z^k + ... +
b_{k+3}z^{k+3} + b_{k+5} z^{k+5}\] instead, or where both functions
include larger order coefficients, or where these depend more on the
values of $\alpha$ and $k$. One may consider, as well, spaces
generated by 3 or more functions, or non-polynomial functions. Each
of these routes complicates the computations but provides hope of
solving more cases than the ones presented here. We also
deliberately ignored some degenerate solutions to some of the linear
systems without exploring them completely. In any case, our proof
does rely on the terms described and the proposed modifications
would likely affect the procedure considerably.

\item[(4)] The main role of the special coefficients $a_4$ and $b_5$
is to ensure a ``register'' of the functions $f_1$ and $f_2$
defining a function $f$ in the subspace $M$. This register was
encoded here in the Taylor coefficients of $f$ using the fact that
the closed span of $\{z^{ks+4}\}_{s \in \N}$ is an invariant
subspace under $z^k$ in $D_\alpha$ (respectively, $\{z^{ks+5}\}_{s
\in \N}$ for $b_5$). This proved useful in Lemma \ref{Lemma1} and
several times after that. If one can find another way of doing this
register without blocking so many coefficients only for that
purpose, they may be able to increase the efficacy of our method to
make it work for $k=4$ or $k=5$. This could be based on some
division lemmas in the style of Lemma 5.2 in \cite{GGR}. Notice that
some steps would need to be reproved like the fact that
Cauchy-Schwarz inequality was strict when applied on Lemma
\ref{Lemma5} and Lemma \ref{Lemma6}. Our method does not seem to go
any further than $k =4$ since we need the matrix $N$ to be at least
of 4 columns wide for nontrivial solutions to exist to the linear
system it induces, unless there are redundancies in the matrix $N$.
Therefore, for the cases $k=2$ and $k=3$ an essential modification
of our method seems necessary.

\item[(5)] Up to an equivalent norm, $D_{-16}$ is the same space as
the weighted Bergman space \[A_{15}^2 = \{ f \in Hol(\D): \|f\|^2 =
\int_{\D} |f(z)|^2(1-|z|^2)^{15}dA(z) < \infty\},\] which is another
example of a $H^2_\omega$ space. One seems inclined to wonder
whether the counterexample presented here still fails to have the
wandering property in this other norm, and if not, whether any other
counterexample may be built in a similar spirit. The failure of the
wandering property for the shift (the case $k=1$) was shown for
$A_{\beta}^2$ for all values of $\beta \geq 4$ in \cite{Netal}, so
it seems plausible that the same threshold can be achieved for other
values of $k$. In general, the same questions make sense for the
standard Bergman-type spaces $A^2_\beta$ (given by weights
$(1-|z|^2)^\beta$, for $\beta
>-1$). Only when $\beta=0$ this coincides with the norm of the
corresponding $D_\alpha$ space ($\alpha= -\beta -1$). Our numerical
results in the previous section (for $D_\alpha$ spaces and some $k
\geq 6$) seem to agree with the results of \cite{Netal} for $k=1$
and the corresponding $A^2_\beta$ spaces.

\end{itemize}

\noindent\textbf{Acknowledgements.} The author is grateful to Eva
Gallardo-Guti\'errez, who provided many forms of support for this
work, while acting as a mentor for the postdoctoral grant funded by
the Severo Ochoa Programme for Centers of Excellence in R\&D
(SEV-2015-0554) at ICMAT. We also acknowledge support from the
Spanish Ministry of Economy and Competitiveness, through grant
MTM2016-77710-P. Finally, we thank W. Ross for useful comments on a
previous version of this article.

\end{document}